\documentclass[oneside,a4paper]{amsart}

\expandafter\let\csname ver@amsthm.sty\endcsname\relax
\let\theoremstyle\relax

\usepackage[utf8]{inputenc}
\usepackage[T1]{fontenc}
\usepackage[english]{babel}
\usepackage{amsmath, amsfonts, amssymb, color}
\usepackage{lmodern}
\usepackage{mathtools}
\usepackage{enumitem}
\usepackage{ifthen}
\usepackage[pdftex,hidelinks,bookmarks=true]{hyperref}
\usepackage{amsthm}
\usepackage[english]{cleveref}
\usepackage[babel]{csquotes}

\usepackage[all,2cell]{xy}  
\UseAllTwocells
\SelectTips{eu}{10}

\usepackage{setspace}
\setdisplayskipstretch{.6}

\newtheorem{lemma}{Lemma}
\newtheorem{theorem}    [lemma]{Theorem}

\newtheorem{proposition}[lemma]{Proposition}
\theoremstyle{definition}

\newtheorem*{remark*}          {Remark}
\newtheorem{question}   [lemma]{Question}
\newtheorem*{question*}  {Question}

\newtheorem*{summary}    {Summary}
\newtheorem*{ack}    {Acknowledgement}

\crefname{lemma}{Lemma}{Lemma}
\crefname{claim}{Claim}{Claim}
\crefname{theorem}{Theorem}{Theorem}
\crefname{corollary}{Corollary}{Corollary}
\crefname{fact}{Fact}{Fact}
\crefname{proposition}{Proposition}{Proposition}
\crefname{definition}{definition}{definition}
\crefname{remark}{remark}{remark}
\crefname{question}{question}{question}
\crefname{example}{example}{example}
\crefname{section}{Section}{Section}
\crefformat{equation}{(#2#1#3)}
\crefformat{enumi}{(#2#1#3)}

\def\cftry#1#2#3{\expandafter\def\csname #1#3\endcsname{{\csname #2\endcsname{#3}}}}

\def\cfH#1{\ifx#1\cfH\else\cftry{H}{mathbb}#1\expandafter\cfH\fi}
\cfH ABCDEFGHIJKLMNOPQRSTUVWXYZ\cfH
\def\cftd#1{\ifx#1\cftd\else\cftry{td}{tilde}#1\expandafter\cftd\fi}
\cftd ABCDEFGHIJKLMNOPQRSTUVWXYZabcdefghijklmnopqrstuvwxyz\cftd
\def\cfcl#1{\ifx#1\cfcl\else\cftry{cl}{mathcal}#1\expandafter\cfcl\fi}
\cfcl ABCDEFGHIJKLMNOPQRSTUVWXYZabcdefghijklmnopqrstuvwxyz\cfcl
\def\cfkk#1{\ifx#1\cfkk\else\cftry{k}{mathfrak}#1\expandafter\cfkk\fi}
\cfkk ABCDEFGHIJKLMNOPQRSTUVWXYZabcdefghijklmnopqrstuvwxyz\cfkk
\def\cfht#1{\ifx#1\cfht\else\cftry{ht}{hat}#1\expandafter\cfht\fi}
\cfht ABCDEFGHIJKLMNOPQRSTUVWXYZabcdefghijklmnopqrstuvwxyz\cfht
\def\cful#1{\ifx#1\cful\else\cftry{ul}{underline}#1\expandafter\cful\fi}
\cful ABCDEFGHIJKLMNOPQRSTUVWXYZabcdefghijklmnopqrstuvwxyz\cful

\let\closure\overline 
\newcommand{\df}{\coloneqq}

\DeclareMathOperator{\End}{End}

\DeclareMathOperator{\GL}{GL}

\makeatletter\let\DH\@undefined\makeatother 
\DeclareMathOperator{\DH}{DH}

\DeclareMathOperator{\Sym}{Sym}

\newcommand{\pleth}[3]{\operatorname{p}(#1,#2;\,#3)}
\DeclareMathOperator{\lrcoeff}{\operatorname{c}}

\DeclareMathOperator{\im}{im}

\let\tmpdet\det\renewcommand{\det}{\tmpdet\nolimits} 
\let\tmpcirc\circ\renewcommand{\circ}{\mathop{\tmpcirc}}

\makeatletter
\def\ifempty#1{\def\@temp{#1}\ifx\@temp\@empty} 
\makeatother
\newcommand{\ifnonempty}[2]{\ifempty{#1}\else#2\fi}
\newcommand{\enclspacing}{}

\newcommand{\cenclose}[7][auto]{%
\ifempty{#1} %
 \ifnonempty{#2}{#2\enclspacing}#3 %
 \ifnonempty{#4}{#7#4#7}#5%
 \ifnonempty{#6}{\enclspacing#6}%
\else\ifthenelse{\equal{#1}{auto}}{%
 \ifthenelse{\equal{#2}{}}{\left.}{\left#2}\enclspacing#3%
 \ifthenelse{\equal{#4}{}}{}{#7\middle #4#7}#5%
 \enclspacing\ifthenelse{\equal{#6}{}}{\right.}{\right#6}
}{
 \ifthenelse{\equal{#2}{}}{}{\csname#1l\endcsname#2}\enclspacing#3%
 \ifthenelse{\equal{#4}{}}{}{#7\csname#1\endcsname#4#7}#5%
 \enclspacing\ifthenelse{\equal{#6}{}}{}{\csname#1r\endcsname#6}
}\fi}
\newcommand{\set}[2][auto]{\enclose[#1]\{{#2}\}}

\newcommand{\abs}[2][auto]{\enclose[#1]|{#2}|}

\newcommand{\gen}[2][auto]{\enclose[#1]\langle{#2}\rangle}

\newcommand{\enclose}[4][auto]{\cenclose[#1]{#2}{#3}{}{}{#4}{}}

\newcommand{\cset}[3][auto]{\cenclose[#1]\{{#2}\lvert{#3}\}\:}

\newcommand{\quot}[3][auto]{\cenclose[#1]{}{#2}{/}{#3}{}{}}
\newcommand{\qq}[3][auto]{#2/\!\!/#3}

\newcommand{\dotter}[3][]{#2#3\ifthenelse{\equal{#1}{}}{}{\widehat{#1}#3}#2}
\newcommand{\pts}[2][]{\dotter[#1]{#2}{\cdots}}
\newcommand{\dts}[2][]{\dotter[#1]{#2}{\ldots}}

\newcommand{\longdashrightarrow}{\mathrel{
\longrightarrow\kern-16pt%
{\color{white}\rule[1.6pt]{1.6pt}{1.6pt}}\kern2pt%
{\color{white}\rule[1.6pt]{1.6pt}{1.6pt}}\kern2pt%
{\color{white}\rule[1.6pt]{1.6pt}{1.6pt}}\kern6pt}}
\newcommand{\cmmt}[3][\circlearrowleft]{\ar@{}[#2]|#3{#1}}
\newcommand{\prar}{\ar@{->>}}
\newcommand{\dtar}{\ar@{..>}}
\newcommand{\dshar}{\ar@{-->}}

\newcommand{\thePolynomial}{P}
\newcommand{\orbitmap}{{\omega}}
\newcommand{\orbitmapq}{\phi}
\newcommand{\orbitmapqn}{{\psi}}

\newcommand{\amsfix}{%
\let\qedold\qedsymbol%
\renewcommand{\qedsymbol}{$\Box$}\qed%
\renewcommand{\qedsymbol}{\qedold}}

\title{A note on normalizations of orbit closures}
\author[J.~Hüttenhain]{Jesko Hüttenhain}
\address{
  Technische Universität Berlin\\
  Straße des 17. Juni 136\\
  10623 Berlin\\
  Germany
}
\email[J.~Hüttenhain]{jesko@math.tu-berlin.de}
\keywords{Geometric complexity theory, birational geometry, classical groups}
\subjclass[2000]{14L24, 13A50}
\thanks{Partially funded by grant \texttt{BU~1371/2-2} of the Deutsche Forschungsgemeinschaft.}

\let\bar\overline 
\newcommand{\normalization}[1]{\widetilde{#1}} 
\renewcommand{\normalization}[1]{\operatorname{N}(\bar{#1})}

\begin{document}
\begin{abstract} 
We give a negative answer to a question by J.M.~Landsberg on the nature of normalizations of orbit closures. A counterexample originates from the study of complex, ternary, cubic forms.
\end{abstract}

\maketitle

\section*{Introduction}

We work over the field of complex numbers and $W$ will always denote a finite-dimensional $\HC$-vector space. Mulmuley and Sohoni \cite{MulSoh01} propose, in their \emph{geometric complexity theory}, to
study the geometry of the orbit closure of certain homogeneous polynomials $P\in\HC[W]$ under the action of $\GL(W)$ by precomposition. Unfortunately, the orbit closures are not normal in the main cases of interest \cite{MR3093509,MR2932001}. When $P$ is a product of linear forms, the normalization of the orbit closure has a representation-theoretic description. This observation motivated J.M.~Landsberg in \cite[Problem~7.19]{La14} to pose the following question:

\begin{question} \label{LandsbergQuestion1} 
\textit{\enquote{Is it true that whenever a $\GL(W)$-orbit closure with reductive stabilizer has an irreducible boundary, the coordinate ring of the normalization of the orbit closure equals the polynomial part of the coordinate ring of the orbit?}}
\end{question}

The remaining undefined terminology will be explained in \cref{partPolynomialPart}.
We show that the orbit of an elliptic curve is an example that answers this question in the negative:

\begin{theorem}
 \label{AnswerToLandsbergQuestion1} Let $W=\HC^3$ and let $\thePolynomial\in\HC[W]_3$ be any form that defines a nonsingular curve in $\HP^2$.
The group $\GL(W)$ acts on $\HC[W]_3$ by precomposition. Let $\Omega$ be the orbit of $\thePolynomial$ under this action, $\bar\Omega$ its Zariski closure and $\nu\colon \normalization\Omega\to \bar\Omega$ the normalization of $\bar\Omega$. 

Then, the stabilizer of $\thePolynomial$ is reductive, the boundary $\partial\Omega\df\bar\Omega\setminus\Omega$ is irreducible and $\HC[\normalization\Omega]$ is not isomorphic to the polynomial part of $\HC[\Omega]$.
\end{theorem}

In \cref{partCubic}, we quote classical results about ternary cubics to deduce that the stabilizer of $\thePolynomial$ is reductive and the boundary of its orbit irreducible. The main point of this note is to verify the other claims of \cref{AnswerToLandsbergQuestion1}, which we will do at the end of \cref{partPolynomialPart}. The third section is dedicated to the proof of an auxiliary result that sheds some further light on the context of \Cref{LandsbergQuestion1}. \Cref{aronhold} illustrates \cref{AnswerToLandsbergQuestion1} by example of the Fermat cubic $\thePolynomial=x^3+y^3+z^3$ and its orbit closure, the Aronhold hypersurface.

\begin{ack}
I thank Peter Bürgisser very much for his many helpful comments about this note. I also extend my gratitude to J.M.~Landsberg for his encouragement and several related discussions. Finally, I am grateful to Christian Ikenmeyer for providing Formula \eqref{formula-ikenmeyer} in \cref{aronhold}.
\end{ack}

\section{Elliptic Curves have Irreducible Boundary}
\label{partCubic}

We let $W=\HC^3$ and consider the space $\HC[W]_3$ of ternary cubics.  $\GL(W)$ acts on this space from the right via $\HC[W]_3\times\GL(W)\to\HC[W]_3$, $(Q,g)\mapsto Q\circ g$. For  $Q\in\HC[W]_3$, we denote by $\Omega_Q\df Q\circ\GL(W)$ the orbit of $Q$. We define $\bar\Omega_Q\df\closure{\Omega_Q}$ to be its Zariski closure and set $\partial\Omega_Q\df\bar\Omega_Q\setminus\Omega_Q$. We say that
 $\thePolynomial\in\HC[W]_3$
is \emph{regular} if its (projective) vanishing set is a smooth plane projective curve. This is equivalent to the condition that for every nonzero $w\in W$, there is a partial derivative of $\thePolynomial$ which does not vanish at $w$. The following is classical and well-known:

\begin{proposition} \label{DegeneracyBehaviourOfTernaryCubic} If $\thePolynomial\in\HC[W]_3$ is regular, it has a finite (hence reductive) stabilizer and the boundary $\partial\Omega_\thePolynomial$ is irreducible.
\end{proposition}
\begin{proof} By Corollary~1 of \cite{MR2035244}, the stabilizer of $\thePolynomial$ is finite. Any finite group is reductive by Maschke's Theorem. A complete diagram of the degeneracy behaviour of all ternary cubic forms can be found in Section~4 of \cite{MR2035244}. Choosing coordinates $\HC[W]=\HC[x,y,z]$, the diagram implies that $\partial\Omega_\thePolynomial$ is equal to the orbit closure of the polynomial $Q\df x^3 - y^2z$, regardless of the choice of $\thePolynomial$. Hence, $\partial\Omega_\thePolynomial=\bar\Omega_Q$ is irreducible: It is the closure of the image of the irreducible variety $\GL(W)$ under the regular map $\GL(W)\to\HC[W]_3$, $g\mapsto Q\circ g$.
\end{proof}

\section{The Polynomial Part of a Module}
\label{partPolynomialPart}

In analogy to the case of homogeneous polynomials, we consider right actions of $\GL(W)$ on an arbitrary, finite-dimensional $\HC$-vector space $\HV$. This action will also be denoted by a \enquote{$\circ$}. To explain the notion \emph{polynomial part}, we require a brief recollection of the representation theory of $\GL(W)$, see \cite{Hum98,Kra85} for some textbooks on the subject. The irreducible representations of $\GL(W)\cong\GL_n(\HC)$ are classified by the semigroup 
\[ \Lambda \df \cset{ \lambda\in\HZ^n }{ \lambda_1\pts\ge\lambda_n } \]
of \emph{dominant weights}. We denote by $\HV(\lambda)$ the irreducible $\GL(W)$-module corresponding to the weight $\lambda\in\Lambda$. There is a partial ordering $\sqsubseteq$ on $\Lambda$ defined as follows: We have $\mu\sqsubseteq\lambda$ if and only if $\mu_i\le\lambda_i$ for all $1\le i\le n$. 
The action of $\GL(W)$ on $\HV(\lambda)$ extends to a morphism $\HV(\lambda)\times\End(W)\to\HV(\lambda)$ of complex varieties if and only if $0\sqsubseteq\lambda$, where $0\in\HZ^n$ denotes the zero vector. 

When $\HV$ is any $\GL(W)$-module, we can decompose it as $\HV\cong\bigoplus_{\lambda\in\Lambda} \HV(\lambda)^{\oplus n_\lambda}$ for certain $n_\lambda\in\HN$. We then write 
\[ \HV_{\sqsupseteq0} \df \bigoplus_{\substack{\lambda\in\Lambda\\\lambda\sqsupseteq0}} \HV(\lambda)^{\oplus n_\lambda} \subseteq \HV. \]
and call this the \emph{polynomial part} of $\HV$. We say that $\HV$ is a \emph{polynomial} $\GL(W)$-module if $\HV=\HV_{\sqsupseteq0}$. We can rephrase \Cref{LandsbergQuestion1} as follows:

\begin{question*} 
\textit{\enquote{Let $\HV$ be a polynomial $\GL(W)$-module and $\thePolynomial\in\HV$. Let $\Omega\subseteq\HV$ be the orbit and $H\subseteq\GL(W)$ the stabilizer group of $\thePolynomial$. Assume that $\partial\Omega$ is irreducible and $H$ is reductive. Let $\nu\colon\normalization\Omega\to\bar\Omega$ be the normalization of the orbit closure of $\thePolynomial$. Is there a $\GL(W)$-module isomorphism $\HC[\normalization\Omega]\cong\HC[\Omega]_{\sqsupseteq0}$?}}
\end{question*}

\begin{theorem} \label{NormalizationClassification} Let $\HV$ be a polynomial $\GL(W)$-module, $\thePolynomial\in\HV$ a point with reductive stabilizer $H\subseteq\GL(W)$. Denote by $\Omega=\thePolynomial\circ\GL(W)$ its orbit and by $\nu\colon \normalization\Omega\to\bar\Omega$ the normalization of its orbit closure. 

There is an injective homomorphism $\iota\colon\HC[\normalization\Omega]\hookrightarrow\HC[\Omega]_{\sqsupseteq0}$ of graded $\HC$-algebras and $\GL(W)$-modules. Futhermore, the following statements are equivalent:
\begin{enumerate}
\item\label{NormalizationClassification:norm} The injection $\iota$ is an isomorphism.
\item\label{NormalizationClassification:cond} For all $a\in\End(W)$ with $\thePolynomial\circ a=0$, we have $0\in\closure{Ha}$. 
\end{enumerate}
If either condition is satisfied, we have $\thePolynomial\circ \End(W)=\bar\Omega$.
\end{theorem}

\Cref{partProof} is dedicated to the proof of this statement. We demonstrate how our main result follows from it:

\begin{proof}[Proof of \Cref{AnswerToLandsbergQuestion1}] By \Cref{DegeneracyBehaviourOfTernaryCubic}, the polynomial $\thePolynomial$ has a reductive stabilizer and its orbit has an irreducible boundary. Let $[w]\in\HP^2$ be any point on the curve defined by $\thePolynomial$, i.e. $w\in W$ is nonzero and $\thePolynomial(w)=0$. Let $a\in\End(W)$ be of rank one such that $\im(a)$ is spanned by $w$. Then, $\thePolynomial\circ a=0$ and since $H$ is a finite group, $\closure{Ha}=Ha$ does not contain the zero map. Hence by \cref{NormalizationClassification}, the two $\GL(W)$-modules $\HC[\normalization\Omega]$ and $\HC[\Omega]_{\sqsupseteq0}$ are not isomorphic.
\end{proof}

\section{Proof of \texorpdfstring{\Cref{NormalizationClassification}}{Theorem~4}}
\label{partProof}

We recall the algebraic Peter-Weyl Theorem, see \cite[27.3.9]{TY05}: 
\begin{align}
\label{PeterWeyl}
 \HC[\GL(W)] &\cong \bigoplus_{\lambda\in\Lambda} \HV(\lambda) \otimes \HV(\lambda)^\ast  
\end{align}
Based on this, we make the following fundamental observation:

\begin{lemma} \label{NormalizationClassificationLemmaEnd} The inclusion $\GL(W)\subseteq\End(W)$ is an open, $\GL(W)$-equivariant immersion of varieties under the operation of $\GL(W)$ acting by multiplication from the left on both affine varieties.
It induces an inclusion of their respective coordinate rings which satisfies $\HC[\End(W)]=\HC[\GL(W)]_{\sqsupseteq0}$. 
\end{lemma}
\begin{proof} By \cite[Formula (6.5.1)]{La11}, we get the last equality in 
\begin{align*} \HC[\End(W)]_d &\cong \HC[W\otimes W^\ast]_d \cong \Sym^d(W\otimes W^\ast) 
\cong \bigoplus_{\substack{0\sqsubseteq\lambda\in\Lambda\\\lambda_1\pts+\lambda_n=d}} \HV(\lambda) \otimes \HV(\lambda)^\ast.
\intertext{By summing over $d$ and applying \cref{PeterWeyl}, we obtain}
 \HC[\End(W)] &\cong \bigoplus_{0\sqsubseteq\lambda\in\Lambda} \HV(\lambda) \otimes \HV(\lambda)^\ast = \HC[\GL(W)]_{\sqsupseteq0}
\end{align*} 
by the definition of the polynomial part.
\end{proof}

Note that the stabilizer $H$ acts on the variety $\End(W)$ by multiplication from the left. Since $\HV$ is a polynomial $\GL(W)$-module, there is a well-defined morphism $\orbitmap\colon\End(W)\to\bar\Omega$, $a\mapsto\thePolynomial\circ a$. For $h\in H$ and $a\in\End(W)$,
\[ \orbitmap(ha)=\thePolynomial\circ (ha)=\thePolynomial\circ h\circ a=\thePolynomial\circ a=\orbitmap(a). \]
Therefore, $\orbitmap$ is an $H$-invariant morphism. Refer to \cite[27.5.1]{TY05} for the following fact:
Since $H$ is a reductive algebraic group, there is an affine quotient variety $\qq{\End(W)}H$ together with a surjective morphism $\pi\colon\End(W)\to \qq{\End(W)}H$ and $\orbitmap$ factors as a morphism $\orbitmapq\colon\qq{\End(W)}H\to\bar\Omega$ such that the following diagram commutes:
\begin{equation}
 \label{NormalizationClassification:eq1}
 \xymatrix{
 \End(W) \cmmt{dr}{(.3)} \prar[d]_-\pi \ar[r]^-\orbitmap & \bar\Omega \\
 \qq{\End(W)}H \dtar[ur]|-*+{\scriptstyle\orbitmapq} & } 
\end{equation} 
Furthermore, the variety $\qq{\End(W)}H$ is a normal variety because $\End(W)$ is normal \cite[27.5.1]{TY05}.

The morphism $\orbitmapq$ is a birational map. Indeed, $\qq{\End(W)}H$ contains $\qq{\GL(W)}H$ as an open subset \cite[27.5.2]{TY05} and the restriction of $\orbitmapq$ to this open subset is the isomorphism $\qq{\GL(W)}H\cong\Omega$, \cite[25.4.6]{TY05}. Since $\Omega$ is an open subset of its closure \cite[21.4.3]{TY05}, this proves that $\orbitmapq$ is generically one to one.

The normalization $\nu\colon \normalization\Omega\to \bar\Omega$ is a surjective, finite morphism of affine algebraic varieties \cite[Proposition 12.43 and Corollary 12.52]{GW10}. 
By the universal property of the normalization \cite[Corollary 12.45]{GW10} there exists a unique morphism $\orbitmapqn\colon\qq{\End(W)}H\to \normalization\Omega$ which completes \cref{NormalizationClassification:eq1} to a commutative diagram:
\begin{equation}
 \label{NormalizationClassification:eq2}
 \xymatrix{
 \End(W) \cmmt{dr}{(.3)} \prar[d]_-\pi \ar[r]^-\orbitmap & \bar\Omega \\
 \qq{\End(W)}H \ar@{->}[ur]|-*+{\scriptstyle\orbitmapq} \ar[r]_-\orbitmapqn & \normalization\Omega  \prar[u]_-\nu \cmmt{ul}{(.3)} } 
\end{equation} 

The morphism $\orbitmapqn$ is dominant and therefore corresponds to an injective ring homomorphism 
\begin{align*} \HC[\normalization\Omega] &\subseteq\HC[\qq{\End(W)}H] =\HC[\End(W)]^H
=(\HC[\GL(W)]_{\sqsupseteq0})^H,
\intertext{due to \cref{NormalizationClassificationLemmaEnd}. Taking $H$-invariants is with respect to the left action of $H$ on $\HC[\GL(W)]$ and considering polynomial submodules is with respect to the right action of $\GL(W)$ on $\HC[\GL(W)]$, so these two operations commute. Hence,}
\HC[\normalization\Omega] &\subseteq (\HC[\GL(W)]^H)_{\sqsupseteq0} = \HC[\Omega]_{\sqsupseteq0}.
\end{align*}

In other words, the polynomial part of the coordinate ring of the orbit of $\thePolynomial$ is the ring of $H$-invariants in $\HC[\End(W)]$, where $H$ is the stabilizer of $\thePolynomial$. 

\begin{summary} There is a commutative diagram of $\GL(W)$-equivariant inclusions of $\HC$-algebras:
\begin{equation}
\label{NormalizationClassification:eqRings}
\xymatrix@C=0.3em@R=2.4em{
 & \HC[\End(W)] &&& \HC[\bar\Omega] \ar[lll] \ar[d] \\
\HC[\Omega]_{\sqsupseteq0} \ar@{=}[r] & \HC[\End(W)]^H \ar[u] &&& \HC[\normalization\Omega] \ar[lll] }
\end{equation}
Here, $\HC[\bar\Omega]$ and $\HC[\End(W)]^H$ have the same quotient field $\HK$. The inclusion
\[ \HC[\normalization\Omega]\subseteq \HC[\End(W)]^H = \HC[\Omega]_{\sqsupseteq0} \]
is an inclusion of integrally closed subrings of $\HK$. 
(By definition, $\HC[\normalization\Omega]$ is the integral closure of $\HC[\bar\Omega]$ in $\HK$.) 
\end{summary} 

We now show that \eqref{NormalizationClassification:norm} of \cref{NormalizationClassification} implies $\thePolynomial\circ \End(W)=\bar\Omega$. Recall \cref{NormalizationClassification:eq2}. The condition $\HC[\normalization\Omega]=\HC[\Omega]_{\sqsupseteq0}$ holds if and only if the morphism $\orbitmapqn$ is an isomorphism. In this case it follows that $\orbitmapq$ is the normalization of $\bar\Omega$. 
Thus, \cref{NormalizationClassification:norm} implies in particular that $\orbitmapq$ is surjective and therefore $\orbitmap$ is surjective, which means $\thePolynomial\circ \End(W)=\bar\Omega$.

We now ask when the inclusion $\HC[\normalization\Omega]\subseteq\HC[\Omega]_{\sqsupseteq0}$ becomes an equality. We will require an auxiliary lemma for the proof. Recall that the algebraic group $\HC^\times=\GL_1$ acts \emph{polynomially} on a variety $X$ if the action morphism $\HC^\times\times X\to X$ lifts to a morphism $\HC\times X\to X$. We will denote this map by a dot, i.e. $(t,x)\mapsto t.x$.

\begin{lemma} \label{NormalizationClassificationLemmaFinite} Let $X$ and $Y$ be affine varieties, each of them equipped with polynomial $\HC^\times$-actions admitting unique
fixed points $0_X\in X$ and $0_Y\in Y$, respectively. Let $\orbitmapq\colon X\to Y$ be a $\HC^\times$-equivariant morphism. Then,  $\orbitmapq^{-1}(0_Y)=\set{0_X}$ if and only if $\orbitmapq$ is finite.
\end{lemma}
\begin{proof} The \enquote{only if} part is \cite[Lemma 7.6.3]{La14}. For the converse, assume that $\orbitmapq$ is finite. Let $x\in X$ be such that $\orbitmapq(x)=0_Y$. Then, $\orbitmapq(t.x)=t.\orbitmapq(x)=t.0_Y=0_Y$ for all $t\in\HC^\times$ and hence, $\HC^\times.x\subseteq\orbitmapq^{-1}(0_Y)$. But $\orbitmapq^{-1}(0_Y)$ is a finite set, therefore $\HC^\times.x$ is finite and irreducible, i.e. a point. This implies $\HC^\times.x=\set{x}$, so $x$ is a fixpoint for the action of $\HC^\times$. It follows that $x=0_X$ by uniqueness of the fixpoint. 
\end{proof}

\cref{NormalizationClassificationLemmaFinite} will be applied to the morphism $\orbitmapq\colon \qq{\End(W)}H\to\bar\Omega$. We therefore study the action of the scalar matrices $\HC^\times\subseteq\GL(W)$ on $\qq{\End(W)}H$ and $\bar\Omega$.  Observe that the morphism $\orbitmapq$ is equivariant with respect to this action. We need to make sure that both varieties have a unique fixpoint in order to make use of \cref{NormalizationClassificationLemmaFinite}.

We first reduce to the case where $\HV$ has a unique fixpoint under the action of all scalar matrices. For this purpose, fix some basis of $W$, so $\GL(W)\cong\GL_n(\HC)$ and let $\HV=\bigoplus_{\lambda\in\HN^n} \HV_\lambda$ be the decomposition of $\HV$ into isotypical components, i.e. $\HV_\lambda$ is a direct sum of irreducible modules of type $\lambda$. Note that the only weights $\lambda$ that appear are in $\HN^n$ because $\HV$ is a polynomial $\GL_n(\HC)$-module. Let $\thePolynomial=\sum_{\lambda\in\HN^n}\thePolynomial_\lambda$ be the corresponding decomposition of $\thePolynomial$, i.e. $\thePolynomial_\lambda\in \HV_\lambda$. Observe that the point $\tilde\thePolynomial \df \thePolynomial - \thePolynomial_0$ has the same stabilizer as $\thePolynomial$, because any element of $\HV_0$ is $\GL(W)$-invariant. Let $\tilde{\HV}\df\bigoplus_{\lambda\ne0} \HV_\lambda$ be the complement of $\HV_0$ in $\HV$. Then, $\Omega=\set{\thePolynomial_0}\times \Omega_{\tilde\thePolynomial} \subseteq \HV_0\times\tilde{\HV}= \HV$ and consequently, 
$\bar\Omega_P =\set{\thePolynomial_0}\times \bar\Omega_{\tilde\thePolynomial} \cong \bar\Omega_{\tilde\thePolynomial}$.
This shows that we may henceforth assume $\HV=\tilde{\HV}$ and $\thePolynomial=\tilde\thePolynomial$. In this situation, the origin $0_\HV\in\HV$ is the only fixpoint under the action of the scalar matrices. Consequently, it is also the only $\HC^\times$-fixpoint in $\bar\Omega$.

On the other hand, $\End(W)$ also has a unique fixpoint with respect to the left action by scalar matrices, namely the zero map which we will denote by $0$. At this point, we require the following lemma to deduce that $\qq{\End(W)}H$ also has a unique fixpoint:

\begin{lemma} \label{NormalizationClassificationLemmaFixpoint} Let $E$ be an affine variety on which $\HC^\times$ acts polynomially with a unique fixpoint $0$. Assume that a reductive group $H$ acts on $E$ from the left such that the actions of $H$ and $\HC^\times$ commute. Then, the quotient $\qq EH$ also has a unique fixpoint under the induced action of $\HC^\times$. 
\end{lemma}

The proof of this lemma is slightly technical and will be given afterwards. Using it, we conclude that $\pi(0)$ is the unique fixpoint in $X\df\qq{\End(W)}H$ and $0_\HV$ is the unique fixpoint in $Y\df\bar\Omega$. The morphism $\orbitmapq\colon X\to Y$ now satisfies the conditions of \cref{NormalizationClassificationLemmaFinite}.

We proceed to prove the equivalence of \eqref{NormalizationClassification:norm} and \eqref{NormalizationClassification:cond}.

We now show $\eqref{NormalizationClassification:norm}\Rightarrow\eqref{NormalizationClassification:cond}$. If \eqref{NormalizationClassification:norm} holds, $\orbitmapqn$ is an isomorphism and $\orbitmapq$ is a normalization of $\bar\Omega$. Therefore, $\orbitmapq$ is a finite morphism. By one direction of \cref{NormalizationClassificationLemmaFinite}, this implies that $\orbitmapq^{-1}(0_\HV)=\set{ \pi(0) }$. In other words, $\orbitmapq(\pi(a))=0_\HV$ implies $\pi(a)=\pi(0)$. We have $\thePolynomial\circ a=\orbitmap(a)=\orbitmapq(\pi(a))$, so $\thePolynomial\circ a=0_\HV$ implies $\pi(a)=\pi(0)$, which is the same as saying that the zero map $0$ is contained in the closure of the $H$-orbit of $a$. This is precisely \eqref{NormalizationClassification:cond}.

For the converse implication, we assume \eqref{NormalizationClassification:cond}. For any $a\in\End(W)$, the condition $0_\HV=\orbitmapq(\pi(a))=\orbitmap(a)$ implies $0\in\closure{Ha}$ by \eqref{NormalizationClassification:cond}. By construction of the GIT quotient, this implies $\pi(a)=\pi(0)$ and hence, $\orbitmapq^{-1}(0_\HV)=\set{\pi(0)}$. The other direction of \cref{NormalizationClassificationLemmaFinite} now states that $\orbitmapq$ is a finite morphism. Any finite morphism is integral \cite[Remark 12.10]{GW10}, so $\orbitmapq$ is an integral birational map from a normal variety $\qq{\End(W)}H$ to $\bar\Omega$. By \cite[Proposition 12.44]{GW10}, it follows that it is the normalization of $\bar\Omega$, so $\orbitmapqn$ is an isomorphism.

\begin{proof}[Proof of \cref{NormalizationClassificationLemmaFixpoint}] We first note that $0$ is a fixpoint for the action of $H$ as well. Indeed, for any $h\in H$ and any $t\in\HC^\times$, we have $t.h.0=h.t.0=h.0$ because the actions commute, so $h.0$ is a fixpoint for the action of $\HC^\times$. By uniqueness, this implies $h.0=0$. As $h\in H$ was arbitrary, $0$ is a fixpoint for the action of $H$.

We will denote by $\pi\colon E\to\qq EH$ the quotient morphism. 
Assume now that $x\df\pi(e)\in\qq EH$ is any fixpoint of the action of $\HC^\times$. Observe that 
\[ \set x=\HC^\times.x=\HC^\times.\pi(e)=\pi(\HC^\times.e), \]
so $\HC^\times.e\subseteq\pi^{-1}(x)$. Since the action of $\HC^\times$ is polynomial, the orbit map lifts to a $\HC^\times$-equivariant morphism $\gamma\colon\HC\to E$ with $\gamma(t)=t.e$ for $t\in\HC^\times$. Since
\[ t.\gamma(0)=\gamma(t\cdot0)=\gamma(0) \]
for all $t\in\HC^\times$, it follows that $\gamma(0)$ is a $\HC^\times$-fixpoint in $E$, so $\gamma(0)=0$. This implies that $0\in\closure{\HC^\times.e}$. 
Because $0$ is a fixpoint for the action of $H$, the set $\set{0}\subseteq E$ is a closed $H$-orbit. By the nature of the GIT-quotient, points of $E$ that share a closed $H$-orbit are mapped to the same point in the quotient. In this case, $e$ and $0$ share the closed orbit $\set{0}$ and it follows that $x=\pi(e)=\pi(0)$. Thus, $\pi(0)$ is the only fixpoint of the action of $\HC^\times$ on $\qq EH$.
\end{proof}

\section{Example: The Aronhold Hypersurface}
\label{aronhold} 

As an example we consider the special case of the Fermat cubic: Let $W=\HC^3$, $\HV=\HC[W]_3=\HC[x,y,z]_3$ and $\thePolynomial\df x^3+y^3+z^3\in\HV$. By \Cref{NormalizationClassification,AnswerToLandsbergQuestion1}, we know that the quotient $\HC[\Omega]_{\sqsupseteq0}/\HC[\normalization{\Omega}]$ exists and that it is a nontrivial $\GL_3$-module, so it decomposes as a direct sum of irreducible $\GL_3$-modules. We will explicitly compute some of the corresponding multiplicities.

Note that this is a special case as the orbit closure $\closure\Omega \subseteq\HV$ is a normal variety \cite[Thm.~1.56]{MR1735271}. This simplifies the calculation because we need not determine the normalization of $\closure\Omega$. Note that if $\thePolynomial$ is a generic regular cubic, $\closure\Omega$ is not normal \cite[Cor.~3.17~(1)]{FundInv}.

$\thePolynomial$ defines the elliptic curve with $j$-invariant equal to zero \cite[Prop.~4.4.7 and eq.~(4.5.8)]{MR1255980}. Its orbit closure $\closure\Omega$ is the hypersurface defined by the Aronhold Invariant $\clA\in\HC[\HV]_4$, see \cite[5.31.1]{MR1213725} and \cite[Remark~7.20]{MR1735271}. Thus,
\begin{align*} 
\HC[\normalization{\Omega}]
& =\HC[\overline\Omega]=
\quot{\HC[\HV]}{\gen{\clA}}.
\intertext{%
We write $\Sym^d \Sym^m \HC^3 =\bigoplus_{\lambda\in\Lambda} \HV(\lambda)^{\pleth{d}{m}{\lambda}}$, where $\pleth{d}{m}{\lambda}$ is known as a \emph{Plethysm} coefficient. 
Furthermore, $\clA$ is a \emph{highest weight vector} of weight~$(4,4,4)$ with respect to the action of $\GL(W)\cong\GL_3(\HC)$. This means that the linear span of the $\GL_3$-orbit of $\clA$ is isomorphic to the irreducible $\GL_3$-module $\HV((4,4,4))$.
Thus, $\gen{\clA}_d=\HC[\HV]_{d-4}\cdot\clA=\bigoplus_{\lambda\in\Lambda} \HV(\lambda)^{\pleth{d-4}{3}{\lambda-(4,4,4)}}$ and}%
\HC[\normalization{\Omega}]_d &\cong \bigoplus_{\lambda\in\Lambda} \HV(\lambda)^{b_\lambda}
\end{align*}
where $b_\lambda =\pleth{d}{3}{\lambda}-\pleth{d-4}{3}{\lambda-(4,4,4)}$.
Note that $b_\lambda$ can be computed with the SCHUR software package (\href{http://schur.sourceforge.net/}{http://schur.sourceforge.net/}).

Denoting by $a_\lambda$ the coefficients such that $\HC[\Omega]_{\sqsupseteq0}=\bigoplus_{\lambda\in\Lambda} \HV(\lambda)^{a_\lambda}$, we are interested in the numbers $m_\lambda\df a_\lambda-b_\lambda$ because
\[ 
\HC[\Omega]_{\sqsupseteq0}/\HC[\normalization{\Omega}] = \bigoplus_{\lambda\in\Lambda} \HV(\lambda)^{a_\lambda-b_\lambda}. 
\]

\begin{figure}[t] \small 
\setlength\tabcolsep{2pt}
\newcommand{\thecolumns}{@{}rrrl@{}}
\newcommand{\thespacing}{\;\;}
\begin{tabular}[t]{\thecolumns}
$a_\lambda$&$b_\lambda$&$m_\lambda$&$\lambda$\\
\hline 
$1$ & $1$ & $0$ & $(4, 2, 0)$ \\
$2$ & $1$ & $1$ & $(6, 0, 0)$ \\ 
\hline 
$1$ & $1$ & $0$ & $(4, 4, 1)$ \\
$1$ & $1$ & $0$ & $(5, 2, 2)$ \\
$2$ & $1$ & $1$ & $(6, 3, 0)$ \\
$2$ & $1$ & $1$ & $(7, 2, 0)$ \\
$1$ & $0$ & $1$ & $(8, 1, 0)$ \\
$3$ & $1$ & $2$ & $(9, 0, 0)$ \\
\hline 
$1$ & $1$ & $0$ & $(6, 4, 2)$ \\
$2$ & $1$ & $1$ & $(6, 6, 0)$ \\
$1$ & $1$ & $0$ & $(7, 3, 2)$ \\
$1$ & $1$ & $0$ & $(7, 4, 1)$ \\
$1$ & $0$ & $1$ & $(7, 5, 0)$ \\
$1$ & $1$ & $0$ & $(8, 2, 2)$ \\
$1$ & $0$ & $1$ & $(8, 3, 1)$ \\
$3$ & $1$ & $2$ & $(8, 4, 0)$ \\
$1$ & $0$ & $1$ & $(9, 2, 1)$ \\
$3$ & $1$ & $2$ & $(9, 3, 0)$ \\
$4$ & $1$ & $3$ & $(10, 2, 0)$ \\
$2$ & $0$ & $2$ & $(11, 1, 0)$ \\
$4$ & $1$ & $3$ & $(12, 0, 0)$ \\
\hline 
$1$ & $1$ & $0$ & $(6, 6, 3)$ \\
$1$ & $1$ & $0$ & $(7, 6, 2)$ \\
$1$ & $1$ & $0$ & $(8, 4, 3)$ \\
$1$ & $1$ & $0$ & $(8, 5, 2)$ \\
$2$ & $1$ & $1$ & $(8, 6, 1)$ \\
$1$ & $0$ & $1$ & $(8, 7, 0)$ \\
$2$ & $2$ & $0$ & $(9, 4, 2)$ \\
$1$ & $0$ & $1$ & $(9, 5, 1)$ \\
$4$ & $1$ & $3$ & $(9, 6, 0)$ \\
$2$ & $1$ & $1$ & $(10, 3, 2)$ \\
$3$ & $1$ & $2$ & $(10, 4, 1)$ \\
$4$ & $1$ & $3$ & $(10, 5, 0)$ \\
$2$ & $1$ & $1$ & $(11, 2, 2)$ \\
$2$ & $0$ & $2$ & $(11, 3, 1)$ \\
$5$ & $1$ & $4$ & $(11, 4, 0)$ \\
$2$ & $0$ & $2$ & $(12, 2, 1)$ \\
$6$ & $1$ & $5$ & $(12, 3, 0)$ \\
$6$ & $1$ & $5$ & $(13, 2, 0)$ \\
$4$ & $0$ & $4$ & $(14, 1, 0)$ \\
$5$ & $1$ & $4$ & $(15, 0, 0)$ \\ 
\hline 
$1$ & $1$ & $0$ & $(6, 6, 6)$ \\
$1$ & $1$ & $0$ & $(8, 6, 4)$ \\
\end{tabular} \thespacing%
\begin{tabular}[t]{\thecolumns}
$a_\lambda$&$b_\lambda$&$m_\lambda$&$\lambda$\\
\hline 
$1$ & $1$ & $0$ & $(8, 8, 2)$ \\
$2$ & $2$ & $0$ & $(9, 6, 3)$ \\
$1$ & $1$ & $0$ & $(9, 7, 2)$ \\
$1$ & $0$ & $1$ & $(9, 8, 1)$ \\
$1$ & $0$ & $1$ & $(9, 9, 0)$ \\
$1$ & $1$ & $0$ & $(10, 4, 4)$ \\
$1$ & $1$ & $0$ & $(10, 5, 3)$ \\
$3$ & $2$ & $1$ & $(10, 6, 2)$ \\
$3$ & $1$ & $2$ & $(10, 7, 1)$ \\
$4$ & $1$ & $3$ & $(10, 8, 0)$ \\
$1$ & $1$ & $0$ & $(11, 4, 3)$ \\
$3$ & $2$ & $1$ & $(11, 5, 2)$ \\
$4$ & $1$ & $3$ & $(11, 6, 1)$ \\
$3$ & $0$ & $3$ & $(11, 7, 0)$ \\
$1$ & $0$ & $1$ & $(12, 3, 3)$ \\
$4$ & $2$ & $2$ & $(12, 4, 2)$ \\
$3$ & $0$ & $3$ & $(12, 5, 1)$ \\
$9$ & $2$ & $7$ & $(12, 6, 0)$ \\
$3$ & $1$ & $2$ & $(13, 3, 2)$ \\
$5$ & $1$ & $4$ & $(13, 4, 1)$ \\
$7$ & $1$ & $6$ & $(13, 5, 0)$ \\
$3$ & $1$ & $2$ & $(14, 2, 2)$ \\
$4$ & $0$ & $4$ & $(14, 3, 1)$ \\
$9$ & $1$ & $8$ & $(14, 4, 0)$ \\
$3$ & $0$ & $3$ & $(15, 2, 1)$ \\
$9$ & $1$ & $8$ & $(15, 3, 0)$ \\
$1$ & $0$ & $1$ & $(16, 1, 1)$ \\
$9$ & $1$ & $8$ & $(16, 2, 0)$ \\
$5$ & $0$ & $5$ & $(17, 1, 0)$ \\
$7$ & $1$ & $6$ & $(18, 0, 0)$ \\
\hline 
$1$ & $1$ & $0$ & $(9, 6, 6)$ \\
$1$ & $1$ & $0$ & $(9, 8, 4)$ \\
$1$ & $1$ & $0$ & $(10, 6, 5)$ \\
$1$ & $1$ & $0$ & $(10, 7, 4)$ \\
$2$ & $2$ & $0$ & $(10, 8, 3)$ \\
$2$ & $1$ & $1$ & $(10, 9, 2)$ \\
$2$ & $1$ & $1$ & $(10, 10, 1)$ \\
$2$ & $2$ & $0$ & $(11, 6, 4)$ \\
$1$ & $1$ & $0$ & $(11, 7, 3)$ \\
$3$ & $2$ & $1$ & $(11, 8, 2)$ \\
$2$ & $0$ & $2$ & $(11, 9, 1)$ \\
$2$ & $0$ & $2$ & $(11, 10, 0)$ \\
$1$ & $1$ & $0$ & $(12, 5, 4)$ \\
\end{tabular} \thespacing
\begin{tabular}[t]{\thecolumns}
$a_\lambda$&$b_\lambda$&$m_\lambda$&$\lambda$\\
\hline 
$4$ & $3$ & $1$ & $(12, 6, 3)$ \\
$4$ & $2$ & $2$ & $(12, 7, 2)$ \\
$5$ & $1$ & $4$ & $(12, 8, 1)$ \\
$6$ & $1$ & $5$ & $(12, 9, 0)$ \\
$1$ & $1$ & $0$ & $(13, 4, 4)$ \\
$2$ & $1$ & $1$ & $(13, 5, 3)$ \\
$6$ & $3$ & $3$ & $(13, 6, 2)$ \\
$5$ & $1$ & $4$ & $(13, 7, 1)$ \\
$8$ & $1$ & $7$ & $(13, 8, 0)$ \\
$3$ & $1$ & $2$ & $(14, 4, 3)$ \\
$5$ & $2$ & $3$ & $(14, 5, 2)$ \\
$8$ & $1$ & $7$ & $(14, 6, 1)$ \\
$9$ & $1$ & $8$ & $(14, 7, 0)$ \\
$1$ & $0$ & $1$ & $(15, 3, 3)$ \\
$6$ & $2$ & $4$ & $(15, 4, 2)$ \\
$6$ & $0$ & $6$ & $(15, 5, 1)$ \\
$13$ & $2$ & $11$ & $(15, 6, 0)$ \\
$5$ & $1$ & $4$ & $(16, 3, 2)$ \\
$8$ & $1$ & $7$ & $(16, 4, 1)$ \\
$12$ & $1$ & $11$ & $(16, 5, 0)$ \\
$4$ & $1$ & $3$ & $(17, 2, 2)$ \\
$6$ & $0$ & $6$ & $(17, 3, 1)$ \\
$13$ & $1$ & $12$ & $(17, 4, 0)$ \\
$5$ & $0$ & $5$ & $(18, 2, 1)$ \\
$13$ & $1$ & $12$ & $(18, 3, 0)$ \\
$1$ & $0$ & $1$ & $(19, 1, 1)$ \\
$12$ & $1$ & $11$ & $(19, 2, 0)$ \\
$8$ & $0$ & $8$ & $(20, 1, 0)$ \\
$8$ & $1$ & $7$ & $(21, 0, 0)$ \\
\hline 
$1$ & $1$ & $0$ & $(10, 8, 6)$ \\
$1$ & $1$ & $0$ & $(10, 9, 5)$ \\
$1$ & $1$ & $0$ & $(10, 10, 4)$ \\
$1$ & $1$ & $0$ & $(11, 8, 5)$ \\
$1$ & $1$ & $0$ & $(11, 9, 4)$ \\
$1$ & $1$ & $0$ & $(11, 10, 3)$ \\
$2$ & $2$ & $0$ & $(12, 6, 6)$ \\
$1$ & $1$ & $0$ & $(12, 7, 5)$ \\
$3$ & $3$ & $0$ & $(12, 8, 4)$ \\
$3$ & $2$ & $1$ & $(12, 9, 3)$ \\
$4$ & $2$ & $2$ & $(12, 10, 2)$ \\
$2$ & $0$ & $2$ & $(12, 11, 1)$ \\
$4$ & $1$ & $3$ & $(12, 12, 0)$ \\
$2$ & $2$ & $0$ & $(13, 6, 5)$ \\
\end{tabular} \thespacing
\begin{tabular}[t]{\thecolumns}
$a_\lambda$&$b_\lambda$&$m_\lambda$&$\lambda$\\
\hline 
$2$ & $2$ & $0$ & $(13, 7, 4)$ \\
$4$ & $3$ & $1$ & $(13, 8, 3)$ \\
$5$ & $2$ & $3$ & $(13, 9, 2)$ \\
$5$ & $1$ & $4$ & $(13, 10, 1)$ \\
$4$ & $0$ & $4$ & $(13, 11, 0)$ \\
$4$ & $3$ & $1$ & $(14, 6, 4)$ \\
$4$ & $2$ & $2$ & $(14, 7, 3)$ \\
$7$ & $3$ & $4$ & $(14, 8, 2)$ \\
$7$ & $1$ & $6$ & $(14, 9, 1)$ \\
$9$ & $1$ & $8$ & $(14, 10, 0)$ \\
$2$ & $1$ & $1$ & $(15, 5, 4)$ \\
$6$ & $3$ & $3$ & $(15, 6, 3)$ \\
$7$ & $3$ & $4$ & $(15, 7, 2)$ \\
$9$ & $1$ & $8$ & $(15, 8, 1)$ \\
$11$ & $1$ & $10$ & $(15, 9, 0)$ \\
$2$ & $1$ & $1$ & $(16, 4, 4)$ \\
$4$ & $1$ & $3$ & $(16, 5, 3)$ \\
$10$ & $3$ & $7$ & $(16, 6, 2)$ \\
$10$ & $1$ & $9$ & $(16, 7, 1)$ \\
$15$ & $2$ & $13$ & $(16, 8, 0)$ \\
$4$ & $1$ & $3$ & $(17, 4, 3)$ \\
$8$ & $2$ & $6$ & $(17, 5, 2)$ \\
$12$ & $1$ & $11$ & $(17, 6, 1)$ \\
$14$ & $1$ & $13$ & $(17, 7, 0)$ \\
$2$ & $0$ & $2$ & $(18, 3, 3)$ \\
$9$ & $2$ & $7$ & $(18, 4, 2)$ \\
$10$ & $0$ & $10$ & $(18, 5, 1)$ \\
$20$ & $2$ & $18$ & $(18, 6, 0)$ \\
$7$ & $1$ & $6$ & $(19, 3, 2)$ \\
$11$ & $1$ & $10$ & $(19, 4, 1)$ \\
$17$ & $1$ & $16$ & $(19, 5, 0)$ \\
$5$ & $1$ & $4$ & $(20, 2, 2)$ \\
$9$ & $0$ & $9$ & $(20, 3, 1)$ \\
$19$ & $1$ & $18$ & $(20, 4, 0)$ \\
$7$ & $0$ & $7$ & $(21, 2, 1)$ \\
$17$ & $1$ & $16$ & $(21, 3, 0)$ \\
$2$ & $0$ & $2$ & $(22, 1, 1)$ \\
$16$ & $1$ & $15$ & $(22, 2, 0)$ \\
$10$ & $0$ & $10$ & $(23, 1, 0)$ \\
$10$ & $1$ & $9$ & $(24, 0, 0)$ \\
\hline 
\multicolumn{4}{l}{$\ldots$ $\ldots$ $\ldots$}
\end{tabular}
\caption{Multiplicities in $\HC[\Omega]_{\sqsupseteq0}/\HC[\normalization{\Omega}]$ for the Fermat cubic, up to degree $8$. Column-wise grouped by degree.}
\label{aronhold-table}
\end{figure} 

It follows from the Peter-Weyl Theorem \cref{PeterWeyl} that $a_\lambda$ is the space of $H$-invariants of $\HV(\lambda)$, where $H\subseteq\GL_3$ is the stabilizer of $\thePolynomial$. It is well-known \cite[Prop.~2.4]{FundInv} that $H$ consists of permutation matrices and diagonal matrices whose diagonal entries are third roots of unity. One can obtain a matrix representation of the canonical projection $\HV(\lambda)\twoheadrightarrow\HV(\lambda)^H$ over the basis of semistandard Young tableaux (SSYT) by symmetrizing each SSYT with respect to $H$ and straightening it \cite[\textsection~7.4]{MR1464693}. The quantity $a_\lambda$ arises as the rank of this matrix. 
We have computed the values of the $m_\lambda\df a_\lambda-b_\lambda$ up to degree $8$, see \Cref{aronhold-table}. 
 
A formula for $a_\lambda$ is more involved than the one for $b_\lambda$.
Advancing methods used in \cite[Section 4.2]{MR2932001} (see also \cite[Section 5.2]{CI}), Ikenmeyer \cite{CI-Personal} determined such a formula: For $\lambda\in\Lambda$, denote by $\abs\lambda\df\lambda_1+\lambda_2+\lambda_3$ the sum of its entries. We have $a_\lambda=0$ unless $\abs\lambda=3d$ for some $d\in\HN$. In this case,
\begin{equation}
\label{formula-ikenmeyer}
a_\lambda = 
\sum_{\substack{\mu\in\Lambda\\\abs\mu=d}}~
\sum_{\substack{\nu_1\dts,\nu_d\in\Lambda,\\\abs{\nu_k}=3\cdot k\cdot\hat\mu_k\\\text{for all $k$}}}
\lrcoeff_{\nu_1\dts,\nu_d}^\lambda \cdot \prod_{k=1}^d \pleth{\hat\mu_k}{3k}{\nu_k},
\end{equation}
where $\lrcoeff_{\nu_1\dts,\nu_d}^\lambda$ denotes the multi-Littlewood-Richardson coefficient and $\hat\mu_k$ denotes the number of times that $k$ appears as an entry of $\mu$. 

\raggedright
\bibliography{biblio}{}
\bibliographystyle{alpha}
\end{document}